\documentclass[12pt,a4paper]{article}
\usepackage{amssymb}
\usepackage{amsmath}
\usepackage{amsthm}
\usepackage{indentfirst}
\usepackage{ot1patch}
\usepackage{psfrag}
\usepackage{rotate}
\usepackage{hyperref}

\theoremstyle{plain}
\newtheorem{tw}{Theorem}[section]
\newtheorem{pr}[tw]{Proposition}
\newtheorem{lm}[tw]{Lemma}
\newtheorem{cor}[tw]{Corollary}

\theoremstyle{definition}

\theoremstyle{remark}
\newtheorem{ex}[tw]{Example}

\newtheorem{rem}[tw]{Remark}

\DeclareMathOperator{\calA}{{\cal A}}
\DeclareMathOperator{\calS}{{\cal S}}
\DeclareMathOperator{\calP}{{\cal P}}
\DeclareMathOperator{\calR}{{\cal R}}
\DeclareMathOperator{\sfF}{\text{\sf F}}
\DeclareMathOperator{\sfq}{\text{\sf q}}

\newcommand{\rpr}{\mathbin{\mbox{\small\rm rpr}}}

\author{Piotr J\k{e}drzejewicz, Miko\l aj Marciniak,\\
\L ukasz Matysiak, Janusz Zieli\'nski\\
\normalsize Faculty of Mathematics and Computer Science\\
\normalsize Nicolaus Copernicus University\\
\normalsize Toru\'{n}, Poland}
\title{On properties of square-free elements\\
in commutative cancellative monoids}
\date{}

\begin{document}

\maketitle

\begin{abstract}
We discuss various square-free factorizations in monoids
in the context of: atomicity, ascending chain condition
for principal ideals, decomposition, and a greatest common
divisor property.
Moreover, we obtain a full characterization of submonoids
of factorial monoids in which all square-free elements
of a submonoid are square-free in a monoid.
We also present factorial properties implying that all atoms
of a submonoid are square-free in a monoid.
\end{abstract}

\section{Introduction}
\label{s1}

Throughout this paper by a monoid we mean a commutative cancellative
monoid. 
We adopt the notation from \cite{GeroldingerHalterKoch}.

\medskip

Let $H$ be a monoid.
We denote by $H^{\times}$ the group of all invertible elements
of $H$.
Two elements $a,b\in H$ are called relatively prime if they have
no common non-invertible divisors, what we denote by $a\rpr b$.
The set of all atoms in $H$ will be denoted by $\calA (H)$.
Recall that an element $a\in H$ is called square-free
if it cannot be presented in the form $a = b^2c$,
where $b, c\in H$ and $b\not\in H^{\times}$.
The set of all square-free elements in $H$ we will denote by
$\calS (H)$.

\medskip

The main motivation of this paper is connected with the following
two properties concerning a submonoid $M\subset H$.
The first one is that all atoms of $M$ are square-free in $H$:
\begin{itemize}
\item[(1.1)] \
$\calA (M)\subset \calS(H)$.
\end{itemize}
The second one is that all square-free elements of $M$ are
square-free in $H$:
\begin{itemize}
\item[(1.2)] \
$\calS (M)\subset \calS(H)$.
\end{itemize}
These properties are related to the famous Jacobian conjecture
(for details see Section \ref{s2}).

\medskip

If $H$ is a factorial monoid and a submonoid $M\subset H$ satisfies
$M^{\times}=H^{\times}$ and $\sfq(M)\cap H=M$, then condition (1.2)
can be expressed in a factorial way
(see \cite{JZanalogs}, Theorem 3.4 -- formulated in terms of rings,
but in fact valid for monoids):
\begin{itemize}
\item[(1.3)]
for every $a\in H$, $b\in\calS(H)$, if $a^2b\in M$, then $a,b\in M$.
\end{itemize}

\medskip

Recall also (see \cite{JZanalogs}, Theorem 3.6) that under these
assumptions a submonoid $M$ satisfying (1.2) is root closed in $H$.
Recently Angerm\"uller showed in \cite{Angermuller}, Proposition 32,
that under the same assumptions a submonoid $M$ satisfying (1.1)
is root closed in $H$.
A submonoid $M\subset H$ is called root closed in $H$ if, for every
$a\in H$ and $n\geq 1$, $a^n\in M$ implies $a\in M$.

\medskip

Recall two questions concerning the conditions (1.1) and (1.2)
in the case of a UFD, stated in \cite{JZapproach}.
We have asked if they are equivalent under some natural assumptions
(like $M^{\times}=H^{\times}$), and if not, can the condition (1.1)
be expressed in a form of factoriality, similarly to (1.3)?

\medskip

In Section \ref{s4} we present a factorial property implying (1.1),
weaker than (1.3), namely:
\begin{itemize}
\item[(1.4)]
for every $a\in H$, $b\in\calS(H)$, if $a^2b\in M$, then $a,ab\in M$.
\end{itemize}
In Theorem \ref{t43} we show that property (1.4) has natural
equivalent forms with respect to various square-free factorizations.

\medskip

In Theorem \ref{t51} we obtain full description of submonoids of
a factorial monoid, satisfying (1.2), as factorial submonoids
generated (up to irreducibles) by any set of pairwise relatively
prime non-invertible square-free elements.
We also obtain the answer to a question, when (1.1) and (1.2)
are equivalent, expressing (1.2) as a conjunction of (1.1)
and the property that any two non-associated atoms of $M$
are relatively prime in $H$.
Moreover, we refer in Theorem \ref{t51} to various square-free
factorizations, in particular equivalence between (1.2) and (1.3)
holds without the assumption $\sfq(M)\cap H=M$.

\medskip

Section \ref{s6} is devoted to properties of radical elements.
Reinhart in \cite{Reinhart} introduced the notions of radical element
and radical factoriality of a monoid.
An element $a\in H$ is called radical if its principal ideal $aH$ is
a radical ideal.
A monoid $H$ is called radical factorial if every element is a product
of radical elements.
As we already observed in \cite{JMZproperties}, Lemma 3.2 b),
every radical element is square-free.
So we have the following diagram of relations on elements of a monoid:
\begin{itemize}
\item[(1.5)] \hfill
$\begin{array}{ccc}
\text{prime}&\Rightarrow &\text{atom} \\
\Downarrow & & \Downarrow \\
\text{radical}&\Rightarrow &\text{square-free} \\
\end{array}$ \hfill $\,$
\end{itemize}
A radical element is an analog of a square-free one in the same way
as a prime element is an analog of an atom.
Moreover, a radical element is a generalization of a prime
in the same way as a square-free element is a generalization of
an atom.

\medskip

How these analogies and generalizations work, we show
in Section~\ref{s6}.
In Propositions \ref{p65} -- \ref{p67} we study the uniqueness
of factorizations.
In Proposition \ref{p64} we prove that in a decomposition monoid
all square-free elements are radical.
Recall that a monoid $H$ is called a decomposition monoid
if every element $a\in H$ is primal, that is,
for every $b,c\in H$ such that $a\mid bc$ there exist
$a_1,a_2\in H$ such that $a=a_1a_2$, $a_1\mid b$ and $a_2\mid c$.
A domain $R$ is pre-Schreier if the multiplicative monoid
$R \setminus \{0\}$ is a decomposition monoid.
The notion of a pre-Schreier domain was introduced by Zafrullah
in \cite{Zafrullah}, see also \cite{BrookfieldRush} and the references
given there.

\medskip

In Sections \ref{s2} and \ref{s7} we discuss square-free
factorizations in monoids in the context of the following
properties: atomicity, ACCP, decomposition, GCD.
We collect all relationships in Proposition \ref{p34}.
This is a generalization and extension of Proposition 1 from
\cite{JMZfactorizations}.
In Section \ref{s7} we consider possible classifications of monoids
with respect to square-free factorizations and we state questions
about existence of monoids.
Some examples are presented in Section \ref{s8}.

\medskip

We refer to the following diagram of relations of monoids:
\renewcommand{\arraycolsep}{4pt}
\renewcommand{\arraystretch}{0.2}
\begin{itemize}
\item[(1.6)] \hfill
$\begin{array}{rcl}
&& \text{BF}\;\;\Rightarrow\;\;\text{ACCP}\;\;\Rightarrow
\;\;\text{atomic} \\
& \mbox{\begin{psfrags}\rotatebox{30}{$\Rightarrow$}\end{psfrags}}
 & \\
\text{factorial} && \\
& \mbox{\begin{psfrags}\rotatebox{-30}{$\Rightarrow$}\end{psfrags}}
 & \\
&& \text{GCD}\;\;\Rightarrow\;\;\text{decomposition}\;\;\Rightarrow
\;\;\text{atoms are primes}
\end{array}$ \hfill $\,$
\end{itemize}
\renewcommand{\arraystretch}{1.2}
Remember that
\begin{itemize}
\item[(1.7)] \hfill
$\text{atomic}\;\wedge\;\text{atoms are primes}\;\;\Rightarrow\;\;
\text{factorial}$ \hfill $\,$
\end{itemize}

\medskip

Finally, in Section \ref{s9} we concern a natural question about
the possible number of square-free elements in a monoid.

\section{Connections with the Jacobian conjecture}
\label{s2}

The Jacobian conjecture, stated by Keller (\cite{Keller}) in 1939
is one of the most important open problems stimulating modern
mathematical research (see \cite{Smale}), with long lists of
false proofs and equivalent formulations.
For more information we refer the reader to van den Essen's book
\cite{vandenEssen}.

\medskip

\noindent
{\bf Jacobian conjecture.}
Let $k$ be a field of characteristic $0$.
For every polynomials $f_1,\dots,f_n\in k[x_1,\dots,x_n]$
with $n\geq 2$, if
\begin{itemize}
\item[(2.1)] \hfill
$\left|\begin{array}{ccc}
\dfrac{\partial f_1}{\partial x_1}&\cdots&
\dfrac{\partial f_1}{\partial x_n}\\
\vdots&&\vdots\\
\dfrac{\partial f_n}{\partial x_1}&\cdots&
\dfrac{\partial f_n}{\partial x_n}
\end{array}\right|\in k\setminus\{0\}$, \hfill $\,$
\end{itemize}
then $k[f_1,\dots,f_n]=k[x_1,\dots,x_n]$.

\medskip

Now, we will describe some topics of an approach to the conjecture
in terms of irreducibility and square-freeness.
For more details we refer the reader to our survey article
\cite{JZapproach}.

\medskip

Under the assumption that $f_1,\dots,f_n$ are algebraically
independent over~$k$, the Jacobian condition (2.1) is equivalent
to any of the following ones
(\cite{BondtYan}, \cite{Jkeller}, \cite{JZanalogs}):
\begin{itemize}
\item[(2.2)]
every atom of $k[f_1,\dots,f_n]$ is square-free
in $k[x_1,\dots,x_n]$,
\item[(2.3)]
every square-free element of $k[f_1,\dots,f_n]$ is square-free
in $k[x_1,\dots,x_n]$.
\end{itemize}
Under the same assumption, the assertion of the conjecture:
$k[f_1,\dots,f_n]=k[x_1,\dots,x_n]$ is equivalent to the following
one (\cite{Bakalarski}, \cite{Adjamagbo}, \cite{Jkeller}):
\begin{itemize}
\item[(2.4)]
every atom of $k[f_1,\dots,f_n]$ is an atom of $k[x_1,\dots,x_n]$.
\end{itemize}
Hence, in particular, the existence of a non-trivial example
for (2.2), where by "non-trivial" we mean "not satisfying (2.4)",
is equivalent to the negation of the Jacobian conjecture.

\medskip

Recall a generalization of the Jacobian conjecture formulated
in \cite{JZanalogs}.

\medskip

\noindent {\bf Conjecture.}
Let $k$ be a field of characteristic $0$.
For every polynomials $f_1,\dots,f_r\in k[x_1,\dots,x_n]$
with $n\geq 2$ and $r\in\{2,\dots,n\}$, if
\begin{itemize}
\item[(2.5)] \hfill
$\gcd\big(\left|\begin{array}{ccc}
\dfrac{\partial f_1}{\partial x_{j_1}}&\cdots&
\dfrac{\partial f_1}{\partial x_{j_r}}\\
\vdots&&\vdots\\
\dfrac{\partial f_r}{\partial x_{j_1}}&\cdots&
\dfrac{\partial f_r}{\partial x_{j_r}}
\end{array}\right|,\;
1\leq j_1<\ldots<j_r\leq n\big)
\in k\setminus\{0\}$, \hfill $\,$
\end{itemize}
then $k[f_1,\dots,f_r]$ is algebraically closed
in $k[x_1,\dots,x_n]$.

\medskip

By Nowicki's characterization (\cite{Nrings}, Theorem~5.5,
\cite{Polder}, Theorem~4.1.5, \cite{DaigleBook}, 1.4)
the assertion above is equivalent to: "$R$ is a ring of constants
for some $k$-derivation of $k[x_1,\dots,x_n]$".

\medskip

Under the assumption that $f_1,\dots,f_r$ are algebraically
independent over~$k$, the generalized Jacobian condition (2.5)
is equivalent to any of the following ones~(\cite{JZanalogs}):
\begin{itemize}
\item[(2.6)]
every atom of $k[f_1,\dots,f_r]$ is square-free
in $k[x_1,\dots,x_n]$,
\item[(2.7)]
every square-free element of $k[f_1,\dots,f_r]$ is square-free
in $k[x_1,\dots,x_n]$.
\end{itemize}

\section{Square-free factorizations in monoids}
\label{s3}

The aim of this section is to recall and extend some observations
from \cite{JMZfactorizations}.
The statements in that paper were formulated for rings,
but the arguments are valid for monoids, since we were working
only with the multiplicative structure of rings.
In particular, Lemma 1 and Lemma 2 e)
of \cite{JMZfactorizations} take the following form.

\begin{lm}
\label{l31}
Let $H$ be a monoid.
If $a\in\calS (H)$ and $a=b_1b_2\dots b_n$, then
$b_1$, $b_2$, $\dots$, $b_n\in\calS (H)$ and $b_i\rpr b_j$ for
$i\neq j$.
\end{lm}

\begin{lm}
\label{l32}
Let $H$ be a decomposition monoid.
If $a_1,\dots,a_n\in\calS (H)$ and $a_i\rpr a_j$ for all $i\neq j$,
then $a_1\ldots a_n\in\calS (H)$.
\end{lm}

As an immediate consequence we obtain.

\begin{cor}
\label{c33}
If $H$ is a decomposition monoid and $a_1,\dots,a_n\in\calA (H)$,
$a_i\nsim a_j$ for $i\neq j$, then $a_1\ldots a_n\in\calS (H)$.
\end{cor}

\medskip

In \cite{JMZfactorizations}, Proposition 1, we considered three
types of square-free factorizations -- (ii), (iii), (iv)
in Proposition \ref{p34} below.
In \cite{JMZfactorizations} we did not consider condition denoted
(i) below as a separate one, as well as atomicity implying it.
Moreover, we considered in \cite{JMZfactorizations}, Proposition 1,
only one type of square-free extraction -- (vi) in Proposition
\ref{p34} below.
Here we add a second type of square-free extraction -- (v)
as easily following from (ii) for an arbitrary monoid.
Finally, implications ${\rm (vi)}\Rightarrow {\rm (ii)}$ and
${\rm (vi)}\Rightarrow {\rm (iv)}$ in \cite{JMZfactorizations},
Proposition 1 b) were formulated for GCD-domains, but the proofs
were based only on \cite{JMZfactorizations}, Lemma 2 e).
This is why implications ${\rm (iii)}\Rightarrow {\rm (ii)}$ and
${\rm (iii)}\Rightarrow {\rm (iv)}$ below hold for arbitrary
decomposition monoids.

\renewcommand{\arraycolsep}{4pt}

\begin{pr}
\label{p34}
Let $H$ be a monoid.
Consider the following conditions:

\medskip

\noindent
{\rm (i)} \
for every $a\in H$ there exist $n\geq 1$ and
$s_1,s_2,\dots,s_n\in\calS (H)$ such that $a=s_1s_2\dots s_n$,

\medskip

\noindent
{\rm (ii)} \
for every $a\in H$ there exist $n\geq 1$ and
$s_1,s_2,\dots,s_n\in\calS (H)$ such that $s_i\mid s_{i+1}$
for $i=1,\dots,n-1$, and $a=s_1s_2\dots s_n$,

\medskip

\noindent
{\rm (iii)} \
for every $a\in H$ there exist $n\geq 1$ and
$s_1,s_2,\dots,s_n\in\calS (H)$ such that $s_i\rpr s_j$
for $i\neq j$, and $a=s_1s_2^2s_3^3\ldots s_n^n$,

\medskip

\noindent
{\rm (iv)} \
for every $a\in H$ there exist $n\geq 0$ and
$s_0,s_1,\dots,s_n\in\calS (H)$ such that
$a=s_0s_1^2s_2^{2^2}\ldots s_n^{2^n}$,

\medskip

\noindent
{\rm (v)} \
for every $a\in H$ there exist $b\in H$ and $c\in\calS (H)$
such that $a=bc$ and $a\mid c^n$ for some $n\geq 1$,

\medskip

\noindent
{\rm (vi)} \
for every $a\in H$ there exist $b\in H$ and $c\in\calS (H)$
such that $a=b^2c$.

\medskip

\noindent
{\bf a)}
The following implications hold:
$$\begin{array}{ccccccc}
&&{\rm (i)}&\Leftarrow &\text{\rm atm} &\Leftarrow &\text{\rm ACCP}
 \\
&&\mbox{\raisebox{-3pt}{$\Uparrow$}} &
\mbox{\begin{psfrags}\rotatebox{-45}{$\Leftarrow$}\end{psfrags}}
&&
\mbox{\begin{psfrags}\rotatebox{-135}{$\Rightarrow$}\end{psfrags}}
& \\
{\rm (ii)}&\Rightarrow &{\rm (iii)}&&{\rm (iv)}&& \\
\Downarrow &&&&\Downarrow && \\
{\rm (v)}&&&&{\rm (vi)}&&
\end{array}$$

\medskip

\noindent
{\bf b)}
If $H$ is a decomposition monoid, then
$$\begin{array}{ccccc}
{\rm (ii)}&\Leftrightarrow &{\rm (iii)}&\Rightarrow &{\rm (iv)}.
\end{array}$$

\medskip

\noindent
{\bf c)}
If $H$ is a GCD-monoid, then
$$\begin{array}{ccccc}
{\rm (ii)}&\Leftrightarrow &{\rm (iii)}&\Leftrightarrow &{\rm (iv)}.
\end{array}$$
\end{pr}

Note that, according to {\rm (v)}, under the assumption $a=bc$
the condition "$a\mid c^n$ for some $n\geq 1$"
is equivalent to "$b\mid c^n$ for some $n\geq 1$".

\medskip

Recall that every radical element is square-free
(\cite{JMZproperties}, Lemma 3.2 b),
so radical factorial monoids studied by Reinhart in \cite{Reinhart}
satisfy condition {\rm (i)}.

\begin{rem}
\label{r35}
The statement that there are (in general) no other implications
than the ones stated above is equivalent to the existence
of the following counter-examples.
\begin{enumerate}
\item
Non-factorial GCD-monoids satisfying:
$$\text{\rm (i)}\wedge\neg\text{\rm (v)},\;
\text{\rm (i)}\wedge\neg\text{\rm (vi)},\;
\text{\rm (v)}\wedge\neg\text{\rm (i)},\;
\text{\rm (v)}\wedge\neg\text{\rm (vi)},\;
\text{\rm (vi)}\wedge\neg\text{\rm (i)},\;
\text{\rm (vi)}\wedge\neg\text{\rm (v)}.$$
\item
A decomposition non-GCD monoid satisfying
$\text{\rm (iv)}\wedge\neg\text{\rm (v)}$.
\item
Non-decomposition monoids satisfying: \
$\text{\rm (ii)}\wedge\neg\text{\rm (vi)}$, \
$\text{\rm (iii)}\wedge\neg\text{\rm (v)}$.
\item
Non-factorial ACCP-monoids satisfying: \
$\neg\text{\rm (iii)}$, \ $\neg\text{\rm (v)}$.
\item
An atomic non-ACCP monoid satisfying
$\neg\text{\rm (vi)}$.
\item
A non-atomic monoid satisfying
$\text{\rm (ii)}$.
\end{enumerate}
\end{rem}

\section{Sufficient conditions for ${\cal A}(M)\subset {\cal S}(H)$}
\label{s4}

In this section we study a factorial property (1.4) implying that
all atoms of a submonoid are square-free in a monoid.
We show that this property is, in general, not a necessary one.
However, it is interesting by itself since it has natural
equivalent forms with respect to several square-free factorizations,
what we obtain in Theorem \ref{t43}.

\begin{pr}
\label{p41}
Let $H$ be a monoid satisfying condition {\rm (vi)}
of Proposition~{\rm\ref{p34}}.
Let $M$ be a submonoid of $H$ such that for every $a\in H$,
$b\in \calS (H)$,
$$a^2b\in M\Rightarrow a,ab\in M.$$
Then $\calA (M)\subset \calS (H)$.
\end{pr}

\begin{proof}
Suppose that there exists some $c\in \calA (M)$ such that
$c\notin \calS (H)$.
Then $c=a^2b$ for some $a\in H$, $b\in \calS (H)$.
Since $a^2b\in M$, then $a,ab\in M$.
Note that $a\notin H^{\times}$, because $c\notin \calS (H)$,
so $a,ab\notin M^{\times}$, a contradiction.
\end{proof}

The converse implication is not valid:

\begin{ex}
\label{e42}
Consider a monoid $H={\mathbb N}^3$ and its submonoid
$M=\langle (1,1,0),$ $(1,0,1) \rangle$.
Then $\calA (M)=\{(1,1,0),(1,0,1)\}$, so $\calA (M)\subset \calS (H)$,
but for $a=(1,0,0)\in H$, $b=(0,1,1)\in \calS (H)$ we have
$2a+b\in M$ and $a,a+b\notin M$.
\end{ex}

Observe that in the above example the monoid $M$ satisfies
$\sfq(M)\cap H=M$, and under this condition properties
(1.3) and (1.4) are equivalent.

\medskip

The most difficult part of Theorem's \ref{t43} proof is
the connection between ${\rm (i)}\Leftrightarrow {\rm (ii)}$ and
${\rm (iii)}\Leftrightarrow {\rm (iv)}\Leftrightarrow {\rm (v)}$,
i.e.\ the equivalence of {\rm (ii)} and {\rm (iii)}.

\begin{tw}
\label{t43}
Let $H$ be a factorial monoid.
Let $M\subset H$ be a submonoid such that $M^{\times}=H^{\times}$.
The following conditions are equivalent:

\medskip

\noindent
{\rm (i)} \
for every $a\in H$ and $b\in\calS (H)$,
$$a^2b\in M\Rightarrow a,ab\in M,$$

\medskip

\noindent
{\rm (ii)} \
for every $n\geq 0$ and $s_0, s_1, \dots , s_n\in \calS (H)$,
$$s_0s_1^2s_2^{2^2}\ldots s_n^{2^n}\!\!\in \!M\Rightarrow
s_is_{i+1}s_{i+2}^2s_{i+3}^{2^2}\ldots s_n^{2^{n-i-1}}\!\!\in \!M,
i=0,\ldots ,n-1, \mbox{and}\; s_n\!\in \!M,$$

\medskip

\noindent
{\rm (iii)} \
for every $n\geq 1$ and $s_1, s_2, \dots , s_n\in \calS (H)$
such that $s_i\rpr_H s_j$ for $i\neq j$,
$$s_1s_2^2s_3^3\ldots s_n^n\in M\Rightarrow s_n, s_{n-1}s_n,
s_{n-2}s_{n-1}s_n, \ldots , s_1s_2\ldots s_n\in M,$$

\medskip

\noindent
{\rm (iv)} \
for every $n\geq 1$ and $s_1, s_2, \dots , s_n\in \calS (H)$
such that $s_i\mid s_{i+1}$ for $i=1, \ldots , n-1$,
$$s_1s_2\ldots s_n\in M\Rightarrow s_{1}, s_2,\ldots , s_n\in M,$$

\medskip

\noindent
{\rm (v)} \
for every $a\in H$ and $b\in\calS (H)$ such that $a\mid b^n$
for some $n\geq 1$,
$$ab\in M\Rightarrow a,b\in M.$$
\end{tw}

\begin{proof}
\noindent
$\text{\rm (i)}\Rightarrow \text{\rm (ii)}$\\
Assume {\rm (i)}.
Consider elements $s_0, \dots , s_n\in\calS (H)$ such that
$s_0s_1^2s_2^{2^2}\ldots s_n^{2^n}\in M$.
Since $\big(s_1s_2^2s_3^{2^2}\ldots s_n^{2^{n-1}}\big)^2s_0\in M$,
from {\rm (i)} we obtain
$$s_1s_2^2s_3^{2^2}\ldots s_n^{2^{n-1}},\;\; \big(s_1s_2^2s_3^{2^2}
\ldots s_n^{2^{n-1}}\big)s_0 \in M.$$
Then, since
$\big(s_2s_3^2s_4^{2^2}\ldots s_n^{2^{n-2}}\big)^2s_1\in M$,
from {\rm (i)} we obtain
$$s_2s_3^2s_4^{2^2}\ldots s_n^{2^{n-2}},\;\; \big(s_2s_3^2s_4^{2^2}
\ldots s_n^{2^{n-2}}\big)s_1 \in M.$$
Continuing, finally we receive:
$$\big(s_1s_2^2s_3^{2^2}\ldots s_n^{2^{n-1}} \big) s_0,\:
\big(s_2s_3^2s_4^{2^2}\ldots s_n^{2^{n-2}} \big)s_1,
\ldots ,\: s_{n-1}s_n^2s_{n-2},\: s_ns_{n-1},\: s_n\in M.$$

\smallskip

\noindent
$\text{\rm (ii)}\Rightarrow \text{\rm (i)}$\\
Assume {\rm (ii)}.
Consider $a\in H, b\in\calS (H)$ such that $a^2b\in M$.
We can express $a$ in the form
$a=s_1s_2^2s_3^{2^2}\ldots s_n^{2^{n-1}}$,
where $s_i\in\calS (H)$ for $i=1, \ldots , n$.
Put $s_0=b$.
Thus we receive:
$$s_0s_1^2s_2^{2^2}\ldots s_n^{2^n}=a^2b\in M.$$
Using the assumption we obtain:
$$s_0s_1s_2^2s_3^{2^2}\ldots s_n^{2^{n-1}},\;
s_1s_2s_3^2s_4^{2^2}\ldots s_n^{2^{n-2}}, \ldots ,\;
s_{n-2}s_{n-1}s_n^2,\; s_{n-1}s_n,\; s_n\in M.$$
We see that $ab=s_0s_1s_2^2s_3^{2^2}\ldots s_n^{2^{n-1}}\in M$.
Moreover:
$$a=s_n(s_{n-1}s_n)\big(s_{n-2}s_{n-1}s_n^2\big)
\big(s_{n-3}s_{n-2}s_{n-1}^2s_n^{2^2}\big)\ldots
\big(s_1s_2s_3^2s_4^{2^2}\ldots s_n^{2^{n-2}}\big)\in M.$$

\smallskip

\noindent
$\text{\rm (ii)}\Rightarrow \text{\rm (iii)}$\\
Assume (ii).
We write ${\lceil x \rceil }$ and ${\lfloor x \rfloor }$
for respectively the ceiling and the floor of a real number $x$.

\medskip

Step I.
{\em If $s_1s_2^2s_3^3\ldots s_n^n\in M$, where
$s_1,\ldots,s_n\in \calS (H)$, $s_i\rpr_H s_j$ for $i\neq j$,
then $s_1s_2s_3^2s_4^2\ldots s_n^{\lceil \frac{n}{2} \rceil },\;
s_2s_3s_4^2s_5^2\ldots s_n^{\lfloor \frac{n}{2} \rfloor }\in M.$}

\medskip

Let $a=s_1s_2^2s_3^3\ldots s_n^n\in M$, where
$s_1,\ldots,s_n\in \calS (H)$, $s_i\rpr_H s_j$ for $i\neq j$.
Then the element $a$ can be presented in the form
$a=t_0t_1^2t_2^{2^2}\ldots t_r^{2^r}$, where
$t_i=s_1^{c^{(1)}_i}\ldots s_n^{c^{(n)}_i}\in \calS (H)$,
$i=0,\ldots , r$ and $k=\sum_{i=0}^r c_i^{(k)}2^i$
with $c_i^{(k)}\in \{0,1\}$, $k=0,1,\ldots ,n$
(see the proof of {\rm (vi)}$\Rightarrow${\rm (ii)}
in \cite{JMZfactorizations}, Proposition 1).
From {\rm (ii)} we get
$t_it_{i+1}t_{i+2}^2t_{i+3}^{2^2}\ldots t_r^{2^{r-i-1}}\in M$,
$i=0,\ldots ,r-1$ and $t_r\in M$.
In particular, $t_0t_1t_2^2\ldots t_r^{2^{r-1}}\in M$.
Moreover:
$$t_1t_2^2t_3^{2^2}\ldots t_r^{2^{r-1}}=
\left(\prod_{i=0}^{r-1}t_it_{i+1}t_{i+2}^2t_{i+3}^{2^2}\ldots
t_r^{2^{r-i-1}}\right)t_r\in M.$$
By the definition of $c^{(j)}_i$, we have
$$s_1s_2s_3^2s_4^2\ldots s_n^{\lceil \frac{n}{2} \rceil }=
t_0t_1t_2^2\ldots t_r^{2^{r-1}}\in M,$$
$$s_2s_3s_4^2s_5^2\ldots s_n^{\lfloor \frac{n}{2} \rfloor }=
t_1t_2^2t_3^{2^2}\ldots t_r^{2^{r-1}}\in M.$$

\smallskip

Step II.
{\em If $s_1s_2^2s_3^3\ldots s_n^n\in M$, where
$s_1,\ldots,s_n\in \calS (H)$, $s_i\rpr_H s_j$ for $i\neq j$,
then $s_1s_2s_3\ldots s_n,\; s_2s_3^2s_4^3\ldots s_n^{n-1}\in M.$}

\medskip

Assume that $s_1s_2^2s_3^3\ldots s_n^n\in M$, where
$s_1,\ldots,s_n\in \calS (H)$, $s_i\rpr_H s_j$ for $i\neq j$.
We prove by induction on $l$ that
$$s_1^{\lceil \frac{1}{2^l} \rceil }
s_2^{\lceil \frac{2}{2^l} \rceil }\ldots
s_{n-1}^{\lceil \frac{n-1}{2^l} \rceil }
s_n^{\lceil \frac{n}{2^l} \rceil },\;\;
s_1^{1-\lceil \frac{1}{2^l} \rceil }
s_2^{2-\lceil \frac{2}{2^l} \rceil }\ldots
s_{n-1}^{n-1-\lceil \frac{n-1}{2^l} \rceil }
s_n^{n-\lceil \frac{n}{2^l} \rceil } \in M.$$
Put $q={\lceil \frac{n}{2^r} \rceil }$.
Then $(q-1)2^r<n\leq q2^r$.
Put $s_i'=s_{(i-1)2^r+1}s_{(i-1)2^r+2}\ldots s_{i2^r}$ for
$i=1,\ldots ,q-1$ and $s_q'=s_{(q-1)2^r+1}s_{(q-1)2^r+2}\ldots s_n$.
Note that $s_1',s_2',\ldots , s_q'\in \calS (H)$ and
$s_i'\rpr_H s_j'$ for $i\neq j$, because
$s_1,\ldots,s_n\in \calS (H)$, $s_i\rpr_H s_j$ for $i\neq j$.
We have
$s_1^{\lceil \frac{1}{2^l} \rceil }
s_2^{\lceil \frac{2}{2^l} \rceil }\ldots
s_{n-1}^{\lceil \frac{n-1}{2^l} \rceil }
s_n^{\lceil \frac{n}{2^l} \rceil }=
s_1'(s_2')^2\ldots (s_q')^q.$
If
$s_1^{\lceil \frac{1}{2^l} \rceil }
s_2^{\lceil \frac{2}{2^l} \rceil }\ldots
s_{n-1}^{\lceil \frac{n-1}{2^l} \rceil }
s_n^{\lceil \frac{n}{2^l} \rceil }\in M$,
then by step I we obtain that
$$s_1^{\lceil \frac{1}{2^{l+1}} \rceil }
s_2^{\lceil \frac{2}{2^{l+1}} \rceil }\ldots
s_{n-1}^{\lceil \frac{n-1}{2^{l+1}} \rceil }
s_n^{\lceil \frac{n}{2^{l+1}} \rceil }
=s_1's_2'(s_3')^2(s_4')^2\ldots (s_q')^{\lceil \frac{q}{2} \rceil }
\in M$$
and
$$s_1^{\lceil \frac{1}{2^l} \rceil -\lceil \frac{1}{2^{l+1}} \rceil }
s_2^{\lceil \frac{2}{2^l} \rceil -\lceil \frac{2}{2^{l+1}} \rceil }
\ldots
s_n^{\lceil \frac{n}{2^l} \rceil -\lceil \frac{n}{2^{l+1}} \rceil }
=s_2's_3'(s_4')^2(s_5')^2\ldots (s_q')^{\lfloor \frac{q}{2} \rfloor }
\in M.$$
If moreover
$s_1^{1-\lceil \frac{1}{2^l} \rceil }
s_2^{2-\lceil \frac{2}{2^l} \rceil }\ldots
s_n^{n-\lceil \frac{n}{2^l} \rceil }\in M$,
then also
$$s_1^{1-\lceil \frac{1}{2^{l+1}} \rceil }
s_2^{2-\lceil \frac{2}{2^{l+1}} \rceil }\ldots
s_n^{n-\lceil \frac{n}{2^{l+1}} \rceil }=$$
$$s_1^{1-\lceil \frac{1}{2^l} \rceil }
s_2^{2-\lceil \frac{2}{2^l} \rceil }\ldots
s_n^{n-\lceil \frac{n}{2^l} \rceil }\cdot
s_1^{\lceil \frac{1}{2^l} \rceil -\lceil \frac{1}{2^{l+1}} \rceil }
s_2^{\lceil \frac{2}{2^l} \rceil -\lceil \frac{2}{2^{l+1}} \rceil }
\ldots
s_n^{\lceil \frac{n}{2^l} \rceil -\lceil \frac{n}{2^{l+1}} \rceil }
\in M.$$

\smallskip

There exists $r \in \mathbb{N}$ such that $2^r > n$.
Then for every $1\leq t\leq n$ we have
$\lceil \frac{t}{2^{r}} \rceil=1$.
Consequently,
$s_1s_2s_3\ldots s_n,\; s_2s_3^2s_4^3\ldots s_n^{n-1}\in M$.

\medskip

Step III.
We prove {\rm (iii)} by induction on $n$.
For $n=1$ it is clear.
Assume the assertion for $n$ and consider
$s_1, s_2, \dots , s_n, s_{n+1}\in \calS (H)$,
$s_i\rpr_H s_j$ for $i\neq j$, such that
$s_1s_2^2s_3^3\ldots s_n^ns_{n+1}^{n+1}\in M$.
By step II we have
$$s_1s_2s_3\ldots s_ns_{n+1}, s_2s_3^2s_4^3\ldots s_n^{n-1}s_{n+1}^n
\in M.$$
Then by the inductive assumption we have
$$s_{n+1}, s_ns_{n+1}, s_{n-1}s_ns_{n+1}, \ldots ,
s_2s_3\ldots s_ns_{n+1}\in M.$$

\smallskip

\noindent
$\text{\rm (iii)}\Rightarrow \text{\rm (ii)}$\\
Assume (iii).
We prove (ii) by induction on $n$.
For $n=0$ it is clear.

\medskip

We assume the assertion for $n$, that is,
if $s_0,s_1,\ldots,s_n \in \calS (M)$,
then $s_0s_1^2s_2^{2^2}\ldots s_n^{2^n}\in M$
implies $s_{n-l}\prod\limits_{j=0}^{l-1} s_{n-l+j+1}^{2^{j}} \in M$
for every $l \in \{0, 1, \ldots , n\}$.

\medskip

We prove the assertion for $n+1$.
Let $a=s_0s_1^2s_2^{2^2}\ldots s_{n+1}^{2^{n+1}}\in M$, where
$s_0,s_1,\ldots ,s_{n+1}\in \calS (H)$.
Then the element $a$ can be presented in the form
$a=t_1t_2^2t_3^3\ldots t_m^m$, where $m=2^{n+2}-1$ and
$t_1,\ldots , t_m\in \calS (H)$, $t_i\rpr_H t_j$, for $i\neq j$
(for details see the proof of ${\rm (ii)} \Rightarrow {\rm (vi)}$
in \cite{JMZfactorizations}, Proposition 1b).
From (iii) we have $t_m, t_{m-1}t_m,\ldots , t_1t_2\ldots t_m \in M$.
Note that $m$ is odd.
Multiplying the elements of the form $t_rt_{r+1}\ldots t_m$
for all odd $r$ we obtain
$t_1t_2t_3^2t_4^2\ldots t_m^{\lceil \frac{m}{2} \rceil}\in M$.
Multiplying the elements of that form for all even $r$ we obtain
$t_2t_3t_4^2t_5^2\ldots t_m^{\lfloor \frac{m}{2} \rfloor}\in M$.
Since
$t_2t_3t_4^2t_5^2\ldots t_m^{\lfloor \frac{m}{2} \rfloor}=
s_1s_2^2s_3^3\ldots s_{n+1}^{2^n}$,
by the inductive assumption we have
$s_{n+1-l}\prod\limits_{j=0}^{l-1} s_{(n+1)-l+j+1}^{2^{j}} \in M$
for $l \in \{0, 1, \ldots , n\}$.
Moreover,
$t_1t_2t_3^2t_4^2\ldots t_m^{\lceil \frac{m}{2} \rceil}=
s_0s_1s_2^2s_3^3\ldots s_{n+1}^{2^n}$,
which gives the assertion for $l=n+1$.

\medskip

\noindent
$\text{\rm (iii)} \Leftrightarrow \text{\rm (iv)}$ follows from
the equivalence of presentations (ii) and (iii) in
Proposition \ref{p34} (for details see \cite{JMZfactorizations},
the proofs of ${\rm (iv)} \Rightarrow {\rm (vi)}$ in Proposition~1a
and ${\rm (vi)} \Rightarrow {\rm (iv)}$ in Proposition~1b).

\medskip

\noindent
$\text{\rm (iv)} \Rightarrow \text{\rm (v)}$\\
Assume (iv).
Consider $a\in H$, $b\in \calS (H)$ such that $a\mid b^n$
for some $n\geq 1$, and $ab\in M$.
Let $a=s_1s_2\ldots s_m$, where $s_1,\ldots ,s_m\in \calS (H)$,
$s_i\mid s_{i+1}$ for $i=1,\ldots ,m-1$.
Then $s_m\mid b^n$, hence $s_m\mid b$, because $s_m\in \calS (H)$.
We have $s_1s_2\ldots s_mb=ab\in M$.
By (iv) we obtain $s_1, s_2,\ldots , s_m, b\in M$, so $a,b\in M$.

\medskip

\noindent
$\text{\rm (v)} \Rightarrow \text{\rm (iv)}$\\
Assume (v).
Let $s_1s_2\ldots s_n\in M$, where $s_1,\ldots ,s_n\in \calS (H)$,
$s_i\mid s_{i+1}$ for $i=1,\ldots ,n-1$.
Put $a=s_1s_2\ldots s_{n-1}$, $b=s_n$.
Then $a\mid b^{n-1}$.
By (v) we have $s_1s_2\ldots s_{n-1}\in M$ and $s_n\in M$,
and the assertion follows by induction.
\end{proof}

\section{Necessary and sufficient conditions for \\
${\cal S}(M)\subset {\cal S}(H)$}
\label{s5}

In this section we obtain a full characterization of submonoids
of a factorial monoid for which all square-free elements
of a submonoid are square-free in a monoid.

\medskip

Let us note that the formulation and the proof of Proposition 4.1
from \cite{JMZproperties} involve only the multiplicative structure
of a domain.
Thus we have the equivalence of the conditions (vi) -- (viii)
of the following Theorem \ref{t51}.
For the same reason implication
$\text{\rm (viii)}\Rightarrow\text{\rm (i)}$ of Theorem \ref{t51}
follows from the proof of implication
$\text{\rm (ii)}\Rightarrow\text{\rm (i)}$ of Theorem 3.4
from \cite{JZanalogs}.

\begin{tw}
\label{t51}
Let $H$ be a factorial monoid.
Let $M\subset H$ be a submonoid such that $M^{\times}=H^{\times}$.
The following conditions are equivalent:

\medskip

\noindent
{\rm (i)} \
$\calS(M)\subset \calS (H)$,

\medskip

\noindent
{\rm (ii)} \
$\calA (M)\subset \calS (H)$ and, for every $a,b\in M$,
$$a\rpr_M b\Rightarrow a\rpr_H b,$$

\medskip

\noindent
{\rm (iii)} \
$\calA (M)\subset \calS (H)$ and, for every $a,b\in \calA (M)$,
$$a\nsim_M b\Rightarrow a\rpr_H b,$$

\medskip

\noindent
{\rm (iv)} \
$M=H^{\times}\times \sfF (B)$, where $B$ is any set of
pairwise relatively prime (in~$H$) non-invertible square-free 
elements of $H$,

\medskip

\noindent
{\rm (v)} \
for every $n\geq 1$ and $s_1,s_2,\dots,s_n\in\calS (H)$
such that $s_i\rpr_H s_j$ for $i\neq j$,
$$s_1s_2^2s_3^3\ldots s_n^n\in M\Rightarrow s_1,s_2,\dots,s_n\in M,$$

\medskip

\noindent
{\rm (vi)} \
for every $n\geq 1$, $k_1,\dots,k_n\geq 0$ and
$q_1,\dots,q_n\in\calA (H)$ such that $q_i\not\sim_H q_j$ for
$i\neq j$,
$$q_1^{k_1}\dots q_n^{k_n}\in M\Rightarrow
q_1^{c^{(1)}_i}\dots q_n^{c^{(n)}_i}\in M\; \mbox{for each}\; i,$$
where $k_j=c^{(j)}_r2^r+\ldots+c^{(j)}_02^0$ for $j=1,\dots,n$,
with $c^{(j)}_i\in\{0,1\}$ for $i=0,\dots,r$.

\medskip

\noindent
{\rm (vii)} \
for every $n\geq 0$ and $s_0,\dots,s_n\in\calS (H)$,
$$s_n^{2^n}\dots s_1^2s_0\in M\Rightarrow s_0,\dots , s_n\in M,$$

\medskip

\noindent
{\rm (viii)} \
for every $a\in H$ and $b\in\calS (H)$,
$$a^2b\in M\Rightarrow a,b\in M.$$
\end{tw}

\begin{proof}
First, observe that $H$ is a BF-monoid and the submonoid $M$
satisfies $M^{\times}=H^{\times}\cap M$, so $M$ is also a BF-monoid,
by \cite{GeroldingerHalterKoch}, Corollary 1.3.3, p.~17.
In particular, $M$ is atomic.

\medskip

\noindent
$\text{\rm (i)}\Rightarrow\text{\rm (iii)}$ \
Assume $\calS (M)\subset \calS (H)$.
Since $\calA (M)\subset \calS (M)$, we have
$\calA (M)\subset \calS (H)$.

\medskip

Suppose that there exist $a,b\in \calA (M)$ such that $a\nsim_M b$
and $a$, $b$ are not relatively prime in $H$.
Then $t=\gcd_H(a,b)\in H\setminus H^{\times}$, so $a=tu$, $b=tv$
for some $u,v\in H$, $u\rpr_H v$.
Since $a,b\in \calA (M)$, we have $a,b\in \calS (H)$,
but $u\mid_H a$, $v\mid_H b$, so $u,v\in \calS (H)$,
and then $uv\in \calS (H)$, because $u\rpr_H v$.

\medskip

Now, we have $ab=t^2uv\not\in \calS (H)$, so $ab\not\in \calS (M)$,
that is, $ab=c^2d$ for some $c\in M\setminus M^{\times}$, $d\in M$.
We may assume that $c\in M\setminus M^{\times}$ is minimal
(with respect to natural length function in $H$)
satisfying the following property:
"there exist $a,b,d\in H$ such that $c\mid_H a,b$ and $ab=c^2d$".
We have $c^2d=t^2uv$, where $uv\in \calS (H)$, so $c\mid_H t$,
because $H$ is factorial, and then $t=cw$ for some $w\in H$.

\medskip

We obtain $a=tu=cwu$, so $uv\in \calS (H)$, since $a\in \calS (H)$.
We have $ac=c^2wu\not\in \calS (H)$, so $ac\not\in \calS (M)$,
hence $ac=e^2h$ for some $e\in M\setminus M^{\times}$, $h\in M$.
Since $e^2h=c^2wu$, where $wu\in \calS (H)$, we infer $e\mid_H c$,
and then also $e\mid_H a$.
We have obtained $e\mid_H a,c$ and $ac=e^2h$, so $e\sim_H c$
by the minimality of $c$.
Then $e\sim_M c$, because $M^{\times}=H^{\times}$.
But $ac=e^2h$, so $a\sim_M eh\sim_M ch$.
Then $a\sim_M c$, since $a\in \calA (M)$
and $c\in M\setminus M^{\times}$.

\medskip

Analogously we show that $b\sim_M c$, so $a\sim_M b$,
a contradiction.

\medskip

\noindent
$\text{\rm (ii)}\Rightarrow\text{\rm (iii)}$ \
It is enough to note that for every $a,b\in \calA (M)$,
$$a\nsim_M b\Rightarrow a\rpr_M b.$$
Namely, if $a,b\in \calA (M)$ are not relatively prime in $M$,
then $a=cd$ and $b=ce$ for some $c\in M\setminus M^{\times}$,
$d,e\in M$, so $d,e\in M^{\times}$ and $a\sim_M b$.

\medskip

\noindent $\text{\rm (iii)}\Rightarrow\text{\rm (ii)}$ \
Assume (iii) and consider elements $a,b\in M$ such that $a\rpr_M b$.
We already know that $M$ is atomic.
Let $a=a_1\dots a_m$ and $b=b_1\dots b_n$ be factorizations
into atoms in $M$.
Since $a\rpr_M b$, for all $i,j$ we have $a_i\nsim_M b_j$,
so $a_i\rpr_H b_j$, but then $a\rpr_H b$.

\medskip

\noindent
$\text{\rm (iii)}\Rightarrow\text{\rm (iv)}$ \
Assume (iii).
Let $B$ be a maximal (with respect to inclusion)
set of pairwise non-associated (in $M$) atoms of $M$.
By (iii) the elements of $B$ are pairwise relatively prime in $H$.
$H$ is a factorial monoid, so $B$ generates a free submonoid.
Since $M$ is atomic and $M^{\times}=H^{\times}$, we obtain
$M=H^{\times}\times \sfF (B)$.

\medskip

\noindent
$\text{\rm (iv)}\Rightarrow\text{\rm (v)}$ \
Assume (iv).
Let $a=s_1s_2^2s_3^3\ldots s_n^n\in M$, where
$s_1,\ldots,s_n\in \calS (H)$, $s_i\rpr_H s_j$ for $i\neq j$.
By (iv), the element $a$ can be presented in the form
$a=ct_1t_2^2t_3^3\ldots t_m^m$ with $c\in H^{\times}$,
$t_i=\prod_{j=1}^{r_i}b_j^{(i)}\in M$, $r_i\geq 0$,
$m\geq n$, and pairwise different all $b_j^{(i)}\in B$.
Since $b_j^{(i)}$ are square-free and pairwise relatively prime
in $H$, then $t_1,\ldots,t_m$ are also square-free and pairwise
relatively prime in $H$.
Finally, for $i=1,\dots,n$ we have $s_i\sim_H t_i$, so $s_i\in M$.

\medskip

\noindent
$\text{\rm (v)}\Rightarrow\text{\rm (vi)}$ \
Assume (v).
Let $a=q_1^{k_1}\dots q_n^{k_n}\in M$, where
$q_1,\dots,q_n\in\calA (H)$, $q_i\not\sim_H q_j$ for $i\neq j$,
and $k_1,\dots,k_n\geq 0$.
Put $m=\max(k_1,\ldots,k_n)$.
For $l=1,\dots,m$ denote $s_l=\prod_{j\colon k_j=l}q_j$.
Then $s_1,s_2,\dots,s_m\in\calS (H)$ and $s_i\rpr_H s_j$
for $i\neq j$.
We have $a=s_1s_2^2\ldots s_m^m$, so $s_1,s_2,\dots,s_m\in M$,
by (v).

\medskip

Now, let $k_j=c^{(j)}_r2^r+\ldots+c^{(j)}_02^0$ for $j=1,\dots,n$,
with $c^{(j)}_i\in\{0,1\}$ for $i=0,\dots,r$.
Note that if $k_{j_1}=k_{j_2}$, then $c^{(j_1)}_i=c^{(j_2)}_i$
for each $i$, so we may denote $d^{(l)}_i=c^{(j)}_i$ for each $j$
such that $k_j=l$, where $l=1,\dots,m$.
Then
$q_1^{c^{(1)}_i}\dots q_n^{c^{(n)}_i}=
s_1^{d^{(1)}_i}\dots s_m^{d^{(m)}_i}\in M$.
\end{proof}

The only type of factorizations from Proposition~{\rm\ref{p34}}
we haven't considered in Theorem~\ref{t43} nor Theorem~\ref{t51}
is {\rm (i)}.
There is no surprise that in this case we obtain a divisor-closed
submonoid.

\begin{pr}
\label{p52}
Let $H$ be a monoid such that each element $a\in H$
can be presented in the form $a=s_1s_2\ldots s_n$,
where $s_1,s_2,\dots,s_n\in\calS (H)$.
Let $M\subset H$ be a submonoid.
The following conditions are equivalent:

\medskip

\noindent
{\rm (i)} \
for every $a,b\in H$,
$$ab\in M\Rightarrow a,b\in M.$$

\medskip

\noindent
{\rm (ii)} \
for every $n\geq 1$ and $s_1,s_2,\dots,s_n\in\calS (H)$,
$$s_1s_2\ldots s_n\in M\Rightarrow s_1,s_2,\dots,s_n\in M.$$
\end{pr}

\section{Radical elements and the uniqueness of factorizations}
\label{s6}

Let $H$ be a monoid.
Recall from \cite{Reinhart} that an element $a\in H$ is called
radical if the principal ideal $aH$ is radical, equivalently,
if for arbitrary $b\in H$ and $n\geq 1$,
$$a\mid b^n\Rightarrow a\mid b.$$
Denote by $\calR (H)$ the set of radical elements of $H$,
and by $\calP (H)$ the set of prime elements.

\medskip

Clearly, every prime element is radical:
$$\calP (H)\subset \calR (H).$$
This is an analog of the fact that every atom is square-free.

\medskip

Note also that every radical element is square-free,
see \cite{JMZproperties}, Lemma~3.2~b),
what is an analog of the fact that a prime element is an atom.

\begin{pr}
\label{p61}
Let $H$ be a monoid. Then
$$\calR (H)\subset \calS (H).$$
\end{pr}

The next lemma completes Lemma~\ref{l31}.

\begin{lm}
\label{l62}
Let $H$ be a monoid and let $a\in\calR (H)$ and $b\in H$.
If $b\mid a$, then $b\in\calR (H)$.
\end{lm}

\begin{proof}
Let $a\in\calR (H)$ and $b\mid a$.
Let $c\in H$ and $b\mid c^n$ for some $n\geq 1$.
By assumption we have $a=bd$, where $d\in H$.
Then $a\mid c^nd^n$ and this implies $a\mid cd$, so $b\mid c$.
\end{proof}

\medskip

In Lemma~\ref{l63} a), b) below we recall Lemma~2 a), d) from
\cite{JMZfactorizations} in terms of monoids.

\begin{lm}
\label{l63}
Let $H$ be a decomposition monoid.

\medskip

\noindent
{\bf a)}
Let $a,b,c\in H$.
If $a\mid bc$ and $a\rpr b$, then $a\mid c$.

\medskip

\noindent
{\bf b)}
Let $a_1,\dots,a_n,b\in H$.
If $a_i\rpr b$ for $i=1,\dots,n$, then $a_1\ldots a_n\rpr b$.

\medskip

\noindent
{\bf c)}
Let $a,b_1,\dots,b_n\in H$.
If $a\mid b_1\ldots b_n$, then there exist $a_1,\dots,a_n\in H$
such that $a=a_1\ldots a_n$ and $a_i\mid b_i$ for $i=1,\dots,n$.

\medskip

\noindent
{\bf d)}
Let $a_1,\dots,a_n\in\calS (H)$, $b\in H$.
If $a_i\rpr a_j$ for $i\neq j$ and $a_i\mid b$ for $i=1,\dots,n$,
then $a_1\ldots a_n\mid b$.
\end{lm}

\begin{proof}
{\bf c)}
Simple induction.

\medskip

{\bf d)}
Induction.
Assume the assertion for $n$.
Consider $a_1,\dots,a_n,a_{n+1}\in\calS (H)$, $a_i\rpr a_j$
for $i\neq j$, and $b\in H$ such that $a_i\mid b$
for $i=1,\dots,n+1$.
Put $a=a_1\ldots a_n$.
Then, by the induction hypothesis, $a\mid b$, so $b=ac$
for some $c\in H$.
Moreover, $a\rpr a_{n+1}$ by b).
Since $a_{n+1}\mid ac$, by a) we obtain $a_{n+1}\mid c$,
and than $aa_{n+1}\mid ac$.
\end{proof}

Now we can prove that in a decomposition monoid every square-free
element is radical.
This is an analog of the fact that in a decomposition monoid
atoms are primes.

\begin{pr}
\label{p64}
Let $H$ be a decomposition monoid. Then
$$\calR (H)=\calS (H).$$
\end{pr}

\begin{proof}
Let $a\in \calS (H)$.
Assume that $a\mid b^n$ for some $b\in H$ and $n\geq 1$.
Then, by Lemma~\ref{l63} c), there exist $a_1,\dots,a_n\in H$
such that $a=a_1\ldots a_n$ and $a_i\mid b$ for $i=1,\dots,n$.
Observe that $a_1,\dots,a_n\in\calS (H)$ and $a_i\rpr a_j$
for $i\neq j$, by Lemma~\ref{l31}, so $a_1\ldots a_n\mid b$
by Lemma~\ref{l63} d).
\end{proof}

In the rest of this section we concern uniqueness properties
of factorizations (ii) -- (iv) and extractions (v), (vi)
from Proposition~\ref{p34}.
In an arbitrary monoid we have the uniqueness of factorization (ii)
and extraction (v) for radical elements.

\begin{pr}
\label{p65}
Let $H$ be a monoid.

\medskip

\noindent
{\bf a)}
For every $r_1,\dots,r_n,t_1,\dots,t_n\in\calR (H)$
such that $r_i\mid r_{i+1}$ and $t_i\mid t_{i+1}$, $i=1,\dots,n-1$,
if
$$r_1r_2\dots r_n\sim t_1t_2\dots t_n,$$
then $r_i\sim t_i$ for $i=1,\dots,n$.

\medskip

\noindent
{\bf b)}
For every $a,c\in H$, $b,d\in \calR (H)$ such that
$a\mid b^m$ and $c\mid d^n$ for some $m,n\geq 1$,
if
$$ab\sim cd,$$
then $a\sim c$ and $b\sim d$.
\end{pr}

\begin{proof}
a) \
Assume that $r_1r_2\dots r_n\sim t_1t_2\dots t_n,$ where
$r_1,\dots,r_n,t_1,\dots,t_n\in\calR (H)$,
$r_i\mid r_{i+1}$ and $t_i\mid t_{i+1}$ for $i=1, \dots,n-1$.
We have $r_n\mid t_1 \dots t_n$, so $r_n\mid t_n^n$.
Since $r_n\in\calR (H)$ we obtain $r_n\mid t_n$.
Analogously, we get $t_n\mid r_n$.
Hence $r_n\sim t_n$ and $r_1\dots r_{n-1}\sim t_1\dots t_{n-1}$.
Then we repeat the above reasoning for $r_{n-1}$ and $t_{n-1}$, etc.

\medskip

\noindent
b) \
Assume that $ab\sim cd$, where $a,c\in H$, $b,d\in\calR (H)$,
$a\mid b^m$ and $c\mid d^n$ for some $m, n\geq 1$.
We see that $b\mid cd$, so $b\mid d^{n+1}$.
Since $b\in\calR (H)$ we obtain $b\mid d$.
Analogously, we get $d\mid b$, so $b\sim d$, and then $a\sim c$.
\end{proof}

In a decomposition monoid we have the uniqueness of factorization
(iii) from Proposition~\ref{p34}.

\begin{pr}
\label{p66}
Let $H$ be a decomposition monoid.
For every $s_1,\dots,s_n$, $t_1$, $\dots$, $t_n\in\calS (H)$
such that $s_i\rpr s_j$ and $t_i\rpr t_j$ for $i\neq j$,
if
$$s_1s_2^2s_3^3\ldots s_n^n\sim t_1t_2^2t_3^3\ldots t_n^n,$$
then $s_i\sim t_i$ for $i=1,\dots,n$.
\end{pr}

\begin{proof}
Assume that
$s_1s_2^2s_3^3\ldots s_n^n\sim t_1t_2^2t_3^3\ldots t_n^n$,
where $s_1,\dots,s_n,t_1,\dots,t_n\in\calR (H)$,
$s_i\rpr s_j$ and $t_i\rpr t_j$ for $i\neq j$.
Put $s'_i=s_i\ldots s_n$, $t'_i=t_i\ldots t_n$ for $i=1,\ldots ,n$.
Then
$$s'_1s'_2\ldots s'_n\sim t'_1t'_2\ldots t'_n.$$
Note that $s'_i, t'_i\in \calS (H)$ for $i=1,\ldots ,n$
by Lemma \ref{l32}.
Since $s'_{i+1}\mid s'_i$ and $t'_{i+1}\mid t'_i$ for
$i=1,\ldots ,n-1$, from Proposition \ref{p65} a) we obtain
$s'_i\sim t'_i$ for $i=1,\ldots ,n$.
Then $s_i\sim t_i$ for $i=1,\dots,n$.
\end{proof}

Finally, recall from \cite{JMZfactorizations}, Proposition 2
(i), (ii), the uniqueness of factorization (iv) and extraction (vi)
for a GCD-monoid.
It was formulated for a~GCD-domain, but the proof is valid for
a GCD-monoid.

\begin{pr}
\label{p67}
Let $H$ be a GCD-monoid.

\medskip

\noindent
{\bf a)}
For every $s_0,s_1,\dots,s_n,t_0,t_1,\dots,t_n\in\calS (H)$,
if
$$s_0s_1^2s_2^{2^2}\ldots s_n^{2^n}\sim
t_0t_1^2t_2^{2^2}\ldots t_n^{2^n},$$
then $s_i\sim t_i$ for $i=0,\dots,n$.

\medskip

\noindent
{\bf b)}
For every $a,c\in H$, $b,d\in \calS (H)$,
if
$$a^2b\sim c^2d,$$
then $a\sim c$ and $b\sim d$.
\end{pr}

\section{Classifications of monoids with respect to
square-free factorizations}
\label{s7}

\renewcommand{\arraycolsep}{3pt}

In this section we show how to organize all the variety of cases
when properties considered in Proposition \ref{p34} hold or do not.
We would like to emphasize two advantages of this situations.
First: it yields mostly non-trivial questions about existence of
7, 19, 24, or even 60 monoids, respectively.
Second: it provides many ways of classifying monoids with respect
to possesing or not different square-free factorizations or
extractions, which may be more subtle than with respect to
irreducible factorizations.

\medskip

There are $7$ possible combinations of logical values
for properties (i)~--~(iv).
$$\begin{array}{|c|c|c|c|} \hline
{\rm (i)}&{\rm (ii)}&{\rm (iii)}&{\rm (iv)} \\ \hline \hline
+&+/- &+&+/- \\ \hline
+&-&-&+/- \\ \hline
-&-&-&- \\ \hline
\end{array}$$

\medskip

We would like to involve the following properties of monoids:
ACCP, atomicity, GCD, decomposition.
We introduce the value of "ACCP/atm" as follows.
$$\begin{array}{|c|c|c|c|c|c|c|c||c|} \hline
\text{ACCP}&\text{atm}&\text{ACCP/atm} \\ \hline \hline
+&+&2 \\ \hline
-&+&1 \\ \hline
-&-&0 \\ \hline
\end{array}$$

\smallskip

Similarly, we introduce the value of "GCD/decomp".
$$\begin{array}{|c|c|c|c|c|c|c|c||c|} \hline
\text{GCD}&\text{decomp}&\text{GCD/decomp} \\ \hline \hline
+&+&2 \\ \hline
-&+&1 \\ \hline
-&-&0 \\ \hline
\end{array}$$

\smallskip

Now, we can collect all possibilities for conditions (i) -- (vi)
in Proposition \ref{p34}, taking into account the properties
mentioned above.
By $1^{\ast}$ below we denote that $1$ as the value of "ACCP/atm"
is possible only when the value of "GCD/decomp" is $0$,
and also $1$ as the value of "GCD/decomp" is possible only when
the value of "ACCP/atm" is $0$.
In the leftmost column we indicate the number of cases for
"ACCP/atm" and "GCD/decomp" with respect to given values of
(i) -- (iv).
In the rightmost column we indicate the number of cases for
extractions (v) and (vi) also with respect to (i) -- (iv).

\smallskip

$$\begin{array}{|c|c|c||c|c|c|c||c|c|c|} \hline
\text{cases}&\text{ACCP/atm}&\text{GCD/decomp}&{\rm (i)}&{\rm (ii)}&
{\rm (iii)}&{\rm (iv)}&{\rm (v)}&{\rm (vi)}&\text{cases} \\
\hline \hline
6&2/1^{\ast}/0&2/1^{\ast}/0&+&+&+&+&+&+&1 \\ \hline
2&1/0&0&+&+&+&-&+&+/- &2 \\ \hline
3&2/1/0&0&+&-&+&+&+/- &+&2 \\ \hline
2&1/0&0&+&-&+&-&+/- &+/- &4 \\ \hline
4&2/1^{\ast}/0&1^{\ast}/0&+&-&-&+&+/- &+&2 \\ \hline
4&1^{\ast}/0&2/1^{\ast}/0&+&-&-&-&+/- &+/- &4 \\ \hline
3&0&2/1/0&-&-&-&-&+/- &+/- &4 \\ \hline
\end{array}$$

\bigskip

Let us extract possible combinations of (i) -- (iv) for:
atomic, ACCP, decomposition and GCD-monoids.
We have:
\begin{itemize}
\item
6 possible combinations for atomic monoids,
$$\begin{array}{|c|c|c|c|} \hline
{\rm (i)}&{\rm (ii)}&{\rm (iii)}&{\rm (iv)} \\ \hline \hline
+&+&+&+/- \\ \hline
+&-&+/-&+/- \\ \hline
\end{array}$$
\item
3 possible combinations for ACCP-monoids,
$$\begin{array}{|c|c|c|c|} \hline
{\rm (i)}&{\rm (ii)}&{\rm (iii)}&{\rm (iv)} \\ \hline \hline
+&+&+&+ \\ \hline
+&-&+/-&+ \\ \hline
\end{array}$$
\item
4 possible combinations for decomposition monoids,
$$\begin{array}{|c|c|c|c|} \hline
{\rm (i)}&{\rm (ii)}&{\rm (iii)}&{\rm (iv)} \\ \hline \hline
+&+&+&+ \\ \hline
+/-&-&-&- \\ \hline
+&-&-&+ \\ \hline
\end{array}$$
\item
3 possible combinations for GCD-monoids.
$$\begin{array}{|c|c|c|c|} \hline
{\rm (i)}&{\rm (ii)}&{\rm (iii)}&{\rm (iv)} \\ \hline \hline
+&+&+&+ \\ \hline
+/-&-&-&- \\ \hline
\end{array}$$
\end{itemize}

\smallskip

There are 24 classes of monoids with respect to properties:
\begin{center}
ACCP, atomicity, GCD, decomposition, (i) -- (iv).
\end{center}
The question if all of them are non-empty is, in our opinion,
of fundamental importance.

\medskip

Extraction (vi) is a basic tool for exploring properties of subrings
connected with square-free elements.
This is why we think it is reasonable to consider whole set
of properties (i) -- (vi).
There arises a question if all combinations of logical values
are possible, i.e., a question about 19 examples.

\medskip

There are 60 classes of monoids with respect to all properties:
\begin{center}
ACCP, atomicity, GCD, decomposition, (i) -- (vi).
\end{center}
We don't think that all of them are non-empty.
It may be true, e.g., that for ACCP-monoids there is
${\rm (ii)}\Leftrightarrow {\rm (v)}$.
Hence, we state a question about 60 examples of monoids.

\section{Some examples}
\label{s8}

\begin{ex}
\label{e81}
Put
$$B_{p,q}=\langle x_1,x_2,x_3,\ldots,y_1,y_2,y_3,\ldots\mid
y_i=x_{i+1}^py_{i+1}^q,i=1,2,3,\ldots\rangle,$$
where $p$, $q$ are positive integers.

Then $B_{p,q}$ is a non-factorial GCD-monoid for any $p$, $q$.

\medskip

\noindent
{\bf a)}
$B_{1,1}$ satisfies all conditions {\rm (i)} -- {\rm (vi)},
in particular, it is a non-atomic monoid satisfying {\rm (ii)},
mentioned in Remark~\ref{r35}.6.

\medskip

\noindent
{\bf b)}
if $q$ is even, then $B_{p,q}$ satisfies {\rm (vi)} and no one of
{\rm (i)} -- {\rm (v)},
in particular, it is a non-factorial GCD-monoid satisfying
$\text{\rm (vi)}\wedge\neg\text{\rm (i)}$ as well as
$\text{\rm (vi)}\wedge\neg\text{\rm (v)}$, mentioned in
Remark~\ref{r35}.1.

\medskip

\noindent
{\bf c)}
if $q$ is odd and $(p,q)\neq (1,1)$, then $B_{p,q}$ satisfies no one
of the conditions {\rm (i)} -- {\rm (vi)}.

\medskip

Monoid $B_{1,1}$ gives an important argument in the discussion
of how property {\rm (i)} extends atomicity in the context of
diagram (1.6):
\renewcommand{\arraycolsep}{4pt}
\renewcommand{\arraystretch}{0.2}
$$\begin{array}{rcl}
&& \text{BF}\;\;\Rightarrow\;\;\text{ACCP}\;\;\Rightarrow
\;\;\text{atomic}\;\;\Rightarrow
\;\;\text{(i)} \\
& \mbox{\begin{psfrags}\rotatebox{30}{$\Rightarrow$}\end{psfrags}}
 & \\
\text{factorial} && \\
& \mbox{\begin{psfrags}\rotatebox{-30}{$\Rightarrow$}\end{psfrags}}
 & \\
&& \text{GCD}\;\;\Rightarrow\;\;\text{decomposition}\;\;\Rightarrow
\;\;\text{atoms are primes}
\end{array}$$
\renewcommand{\arraystretch}{1.2}

\noindent 
Namely, we loose connection with the lower line of the diagram
since $B_{1,1}$ satisfies the strongest one -- {\rm GCD} --
and is not factorial, so in general the conjunction of {\rm (i)}
and {\rm GCD} does not imply factoriality.
\end{ex}

\begin{ex}
\label{e82}
Let ${\mathbb Q}_{\geq 0}$ denote the set of all non-negative
rational numbers.
$H=({\mathbb Q}_{\geq 0},+)$ is a GCD-monoid, because
$\gcd (a,b)=\min \{a,b\}$ for all $a,b\in H$.
It satisfies condition {\rm (vi)}, because for any $a\in H$
we have $a=\frac{a}{2}+\frac{a}{2}+0$ and $0\in \calS (H)$.
However it do not satisfy any of conditions {\rm (i)} -- {\rm (iv)},
because $\calS (H)=\{0\}$ and $0+\ldots +0\neq a$ for $a\neq 0$.
Neither condition {\rm (v)}, because if $c\in \calS (H)$,
then $c=0$ and $0+\ldots +0$ is not divisible in
$({\mathbb Q}_{\geq 0},+)$ by a non-zero $a$
(here $a\mid b$ iff $a\leq b$).
Clearly $H$ is also non-factorial.
\end{ex}

\begin{ex}
\label{e83}
For a non-negative integer $k$ we denote by
${\mathbb N}_{\geq k}$ the set
of integers greater or equal to $k$.
Then $H=({\mathbb N}_{\geq 2}\cup\{0\},+)$
is not a decomposition monoid, since $\calA (H)=\{2,3\}$
and $\calP (H)=\emptyset $. See also section \ref{s8}.
\end{ex}

\begin{ex}
\label{e84}
Let $L$ and $F$ be fields such that $L\subset F$.
Consider $T=L+xF[x]$.
Then the atoms of the ring $T$ are known:
\end{ex}

\begin{tw}{\em (\cite{AndAndZaf2}, Theorems 2.9 and 5.3)}.\\
$T$ is half-factorial domain and
$\calA (T)=\{ax, a\in F\}\cup
\{a(1+xf(x)),a\in L, f\in F[x], 1+xf(x)\in\calA (F[x])\}.$
\end{tw}

We can also determine all the square-free elements of $T$:

\begin{pr}
\label{p89}
Let $T=L+xF[x]$, and $f\in F[x]$.
Then $f\in\calS (T)$ iff $f\in\calS (F[x]) \wedge f(0)\in L$.
\end{pr}

\begin{proof}
Suppose that $f\notin\calS (F[x])$ or $f(0)\notin L$.
If $f(0)\notin L$, then $f\notin T$, so $f\notin\calS (T)$.
Now, assume that $f\notin\calS (F[x])$.
Then $f=g^2h$, where $g\in F[x]\setminus F, h\in F[x]$.
Let $g=a_nx^n+\ldots +a_1x+a_0, h=b_mx^m+\ldots +b_1x+b_0$.
We have $f=(a_nx^n+\ldots +a_1x+a_0)^2(b_mx^m+\ldots +b_1x+b_0)$.
Then
$f=\Big(\dfrac{a_n}{a_0}x^n+\ldots + \dfrac{a_1}{a_0}x+1\Big)^2
(b_ma_0^2x^m+\ldots +b_1a_0^2x+b_0a_0^2)$,
where $b_0a_0^2=f(0)\in L$.

Now, suppose that $f\notin\calS (T)$.
If $f\notin T$, then $f(0)\notin L$.
Now, assume that $f\in T$.
Then we have $f=g^2h$, where $g\in T\setminus L, h\in T$.
This implies $g\in F[x]\setminus F, h\in F[x]$.
\end{proof}

In particular, if $T=\mathbb{R}+x\mathbb{C}[x]$, then\\
$\calA (T)=\{ax, a\in\mathbb{C}\}\cup
\{a(1+bx), a\in\mathbb{R}, b\in\mathbb{C}\setminus\{0\}\}$\\
$\calS (T)=\{1\}\cup\{ax, a\in\mathbb{C}\}\cup
\{a(1+bx), a\in\mathbb{R}, b\in\mathbb{C}\setminus\{0\}\}\cup
\{x(a+bx), a\in\mathbb{R}, b\in\mathbb{C}\setminus\{0\}\}$.

\smallskip

Using Proposition \ref{p89} we easily verify that $L+xF[x]$ fulfills
{\rm (i)} -- {\rm (vi)}.

\smallskip
If $F$ and $L$ are finite fields and it is a proper extension,
then $L+xF[x]$ is a non-factorial ACCP domain
(see \cite{AndAndZaf1}, \cite{EftekhariKhorsandi}).

\section{The number of square-free elements of a reduced monoid}
\label{s9}

It is obvious that an arbitrary non-negative integer can be
the number of atoms of a monoid.
For example it can be the number of its free generators.
In a group every element is square-free, since there is no
non-invertible element.
Hence, any positive integer can be the number of square-free
elements of a monoid.
It is not such obvious, but still true, that an arbitrary positive
integer can be the number of square-free elements of a reduced
monoid.
It also remains valid if we assume that this reduced monoid is
cancellative.

For integers $a,b$ such that $a\geq b$, we define
$[a,b]=\{c\in {\mathbb Z}; a\leq c\leq b\}$, that is,
the set of all consecutive integers from $a$ to $b$.

\begin{tw}
\label{t91}
Let $n$ be a positive integer.
Then there exists a reduced cancellative monoid $H$ such that
$\#\calS (H)=n$.
\end{tw}

\begin{proof}
Let $m$ be an integer $\geq 2$. Consider a monoid
$$H={\mathbb N}_{\geq 2m}\cup \{0\}\cup \{m\}$$
with the operation of addition.

Clearly $\calA (H)=\{m\}\cup [2m+1,3m-1]$ and $\#\calA (H)=m$.
Then $\calS (H)=\{0,m\} \cup [2m+1,3m-1]\cup [3m+1,4m-1]$
and consequently $\#\calS (H)=2m$.

Now let $m$ be an integer $\geq 3$ and consider a monoid
$$H={\mathbb N}_{\geq 2m-1}\cup \{0\}\cup \{m\}.$$
In this case $\calA (H)=\{m,2m-1\}\cup [2m+1,3m-2]$ and
$\#\calA (H)=m$.
Then $\calS (H)=\{0,m,2m-1\} \cup [2m+1,3m-1]\cup [3m+1,4m-3]$
and finally $\#\calS (H)=2m-1$.

So far we have proved the assertion for $n\geq 4$.
If $n=1$ we can take $H=\{0\}$.
If $n=2$ we may consider $H={\mathbb N}_{\geq 0}$.
If $n=3$ we can take $H={\mathbb N}_{\geq 2}\cup \{0\}$.
\end{proof}

Note that the proof could not be based solely on the monoids
of the form $H_k={\mathbb N}_{\geq k}\cup \{0\}$,
because $\#\calS (H_k)$ grows faster than $k$.

\end{document}